\pdfoutput=1
\RequirePackage{ifpdf}
\ifpdf 
\documentclass[pdftex]{sigma}
\else
\documentclass{sigma}
\fi

\usepackage{bm}
\usepackage{wrapfig}

\usepackage{bm}
\usepackage[all]{xy}
\usepackage{amsbsy}

\numberwithin{equation}{section}

\newtheorem{Theorem}{Theorem}[section]
\newtheorem*{Theorem*}{Theorem}

\newtheorem{Proposition}[Theorem]{Proposition}
 { \theoremstyle{definition}
\newtheorem{Definition}[Theorem]{Definition}

\newtheorem{Example}[Theorem]{Example}
 }

\newcommand{\beq}{\begin{equation}}
\newcommand{\eeq}{\end{equation}}
\newcommand{\bea}{\begin{eqnarray*}}
\newcommand{\eea}{\end{eqnarray*}}
\newcommand{\beqa}{\begin{eqnarray}}
\newcommand{\eeqa}{\end{eqnarray}}

\def\bR{{\mathbb{R}}}

\def\bZ{{\mathbb{Z}}}

\newcommand{\calE}{{\mathcal E}}

\newcommand{\calL}{{\mathcal L}}
\newcommand{\calM}{{\mathcal M}}
\newcommand{\calN}{{\mathcal N}}
\newcommand{\calO}{{\mathcal O}}

\newcommand{\sbv}[2]{{\{{{#1},{#2}}\}}}

\newcommand{\courant}[2]{{[{{#1},{#2}}]_D}}

\newcommand{\bracket}[2]{\langle #1,\,#2\rangle}

\newcommand{\inner}[2]{{({{#1},{#2}})}}

\newcommand{\proj}{\operatorname{pr}}

\newcommand{\rd}{\mathrm{d}}

\newcommand{\uJ}{{\underline{J}}{}}

\newcommand{\momega}{{\omega}}
\newcommand{\gomega}{{\omega}_{\rm grad}}
\newcommand{\ggomega}{{\widetilde{\omega}}_{\rm grad}}
\newcommand{\tQ}{{\widetilde{Q}}}
\newcommand{\Thetan}{{\widetilde{\Theta}}}
\newcommand{\phiL}{{\phi_{\rm Lie}}}

\newcommand{\hmu}{{\widehat{\mu}}}

\newcommand{\rde}{{{}^E \rd}}

\newcommand{\hrho}{{\widehat{\rho}}}

\newcommand{\QLie}{{Q_{\mathrm{Lie}}}}
\newcommand{\Qadm}{{Q_{\mathrm{comp}}}}

\begin{document}

\allowdisplaybreaks

\newcommand{\arXivNumber}{2302.08193}

\renewcommand{\PaperNumber}{025}

\FirstPageHeading

\ShortArticleName{Compatible $E$-Differential Forms on Lie Algebroids over (Pre-)Multisymplectic Manifoldss}

\ArticleName{Compatible $\boldsymbol{E}$-Differential Forms on Lie Algebroids\\ over (Pre-)Multisymplectic Manifolds}

\Author{Noriaki IKEDA}

\AuthorNameForHeading{N.~Ikeda}

\Address{Department of Mathematical Sciences, Ritsumeikan University, Kusatsu, Shiga 525-8577, Japan}
\Email{\href{mailto:nikeda@se.ritsumei.ac.jp}{nikeda@se.ritsumei.ac.jp}}
\URLaddress{\url{https://www.ritsumei.ac.jp/~nikeda/}}

\ArticleDates{Received November 13, 2023, in final form March 27, 2024; Published online March 31, 2024}

\Abstract{We consider higher generalizations of both a (twisted) Poisson structure and the equivariant condition of a momentum map on a symplectic manifold. On a Lie algebroid over a (pre-)symplectic and (pre-)multisymplectic manifold, we introduce a Lie algebroid differential form called a compatible $E$-$n$-form. This differential form satisfies a~compatibility condition, which is consistent with both the Lie algebroid structure and the (pre-)(multi)symplectic structure. There are many interesting examples such as a~Poisson structure, a twisted Poisson structure and a twisted $R$-Poisson structure for a pre-$n$-plectic manifold. Moreover, momentum maps and momentum sections on symplectic manifolds, homotopy momentum maps and homotopy momentum sections on multisymplectic manifolds have this structure.}


\Classification{53D17; 	53D20; 58A50}

\section{Introduction}

A Lie algebroid is a generalization of both a Lie algebra and a space of vector fields on a tangent bundle over a smooth manifold.
It is not only important in differential geometry, especially in Poisson geometry, but also has appeared in physics like a gauge theory, string theory and a~sigma model.

In this paper, we consider a compatible geometric structure
between a section of the exterior algebra of $E^*$ with a Lie algebroid and a (pre-)symplectic and a (pre-)multisymplectic structure on the base manifold $M$.
It is called the \textit{compatible} condition.
Here $E^*$ is a dual space of a~Lie algebroid $E$ over and a (pre-)symplectic and a (pre-)multisymplectic smooth manifold $M$.
A~section of the exterior algebra $\wedge^{\bullet} E^*$ is called an \textit{$E$-differential form}.

A Poisson structure is an important geometric structure inspired
by the Hamiltonian mechanics. A deformation of the Poisson structure by adding
a closed $3$-form $H$ is called the twisted Poisson structure \cite{Klimcik:2001vg, Severa:2001qm}.
We consider a bivector field $\pi \in \Gamma\bigl(\wedge^2 TM\bigr)$ and a closed $3$-form $H \in \Omega^3(M)$ on a smooth manifold $M$.
If $(\pi, H)$ satisfy
\begin{align}
 \frac{1}{2}[\pi, \pi]_S
= \bigl\langle\otimes^{3} \pi,H\bigr\rangle,
\label{tPoisson1intro}
\end{align}
$(\pi, H)$ is called a twisted Poisson structure, where
$[-,-]_S$ is the Schouten bracket on the space of multivector fields and
$\bracket{-}{-}$ is the pairing of $TM$ and $T^*M$.
The condition of the twisted Poisson structure is regarded as a compatibility condition between a Lie algebroid and a pre-$2$-plectic form.
To consider a generalization to general pre-$n$-plectic form is natural. However, a~generalization is not straightforward
since a simple generalization of $\pi$ and $H$ in equation~\eqref{tPoisson1intro} to a multivector field
and a pre-$n$-plectic form does not give a good condition.

A twisted Poisson structure was discovered in the analysis of a sigma model called a twisted Poisson (or a WZ Poisson) sigma model \cite{Klimcik:2001vg}.
Recently, Chatzistavrakidis proposed a new higher dimensional topological sigma model with WZ term whose target space was a~Poisson mani\-fold~\cite{Chatzistavrakidis:2021nom, Chatzistavrakidis:2022qhf}.
If and only if a multivector field on the target space is satisfied a geometric condition, the topological sigma model is consistent.
A generalization to general Lie algebroid setting
is formulated in the paper \cite{Ikeda:2021rir}.
Consistency conditions of these sigma models
imposes a geometric condition of an $E$-differential form with
a Lie algebroid structure and a~(pre-)(multi)symplectic form.
In this paper, we investigate this geometric structure mathematically.

Another motivation is a generalization of a momentum map
on a symplectic manifold. A~momentum map is defined on a symplectic manifold
$M$ with a Lie group $G$ action.
Let $\mathfrak{g}$ be a~Lie algebra of the Lie group $G$.
Recently, a generalization of a momentum map theory to a Lie algebroid setting
has been proposed \cite{Blohmann:2018} inspired by analysis in
\cite{Blohmann:2010jd, Kotov:2016lpx}.
A momentum map is~${\mu\colon M \rightarrow \mathfrak{g}^*}$ is generalized to
a section of the dual of a vector bundle $\mu \in \Gamma(E^*)$ called
a momentum section, where $E$ is a Lie algebroid.
If a vector bundle is the direct product $E = M \times \mathfrak{g}$,
it reduces to a momentum map.

Conditions of the momentum map
are generalized.
In this paper, we concentrate on one compatibility condition called
the bracket-compatibility condition, \[{}^E \rd \mu
= - \iota_{\rho}^{2}{\omega},\] where $\omega$ is a~symplectic form, $\rho$ is the anchor map of a Lie algebroid, and ${}^E \rd$ is the Lie algebroid differential.
This condition looks to be curious at first, however it has natural
good properties.
Moreover we claim that the condition is
unified with the twisted Poisson structure,
and generalized to higher algebroids and multisymplectic settings.
A generalization of the momentum section to a~(pre)-multisymplectic manifold
called a homotopy momentum section has be proposed in \cite{Hirota:2021isx}. Our condition for an $E$-$n$-form in this paper is consistent with the definition of the homotopy momentum section.

A twisted Poisson structure is known as an example of Dirac structures \cite{Severa:2001qm}.
A momentum map on a symplectic manifold and a homotopy momentum map \cite{Fregier:2013dda} on a multisymplectic manifold have reinterpretations in terms of higher algebroids \cite{MitiZambon}.
We show that a compatible $E$-$n$-form is an example of a higher Dirac structure.
Moreover, a compatible $E$-$n$-form can be formulated as a Q-manifold
and a QP-manifold. These descriptions are connected to constructions of topological sigma models. One of constructions is so called AKSZ sigma models \cite{Alexandrov:1995kv, Ikeda:2012pv, Roytenberg:2006qz}.

We claim that the above structure is interesting and worth to analyze, since
it has many examples already listed above.
Generalizations of momentum maps such as momentum sections on a Lie algebroid,
homotopy momentum maps \cite{Fregier:2013dda} and homotopy momentum sections \cite{Hirota:2021isx}
on multisymplectic manifolds are examples of our geometric structure.

This paper is organized as follows.
In Section \ref{sec2}, a Lie algebroid and related notion are prepared and
notation is fixed.
In Section~\ref{sec:}, a compatible condition is defined and basic properties are analyzed.
In Section~\ref{sec:Exam}, examples are listed.
In Section~\ref{HDirac}, the interpretation in terms of the higher Dirac structure is discussed.
In Section~\ref{Qmfd}, the compatible condition is formulated as a~Q-manifold.
In Section~\ref{Qmfdtwist}, we reformulate the compatible condition
as a (twisted) QP-manifold.

\section{Preliminary}\label{sec2}
In this section, we explain the background geometry of this paper.
A Lie algebroid and Lie algebroid differentials are introduced.

\subsection{Lie algebroids}

\begin{Definition}
Let $E$ be a vector bundle over a smooth manifold $M$.
A Lie algebroid $(E, \rho, [-,-])$ is a vector bundle $E$ with
a bundle map $\rho\colon E \rightarrow TM$ called the anchor map,
and a Lie bracket
$[-,-]\colon \Gamma(E) \times \Gamma(E) \rightarrow \Gamma(E)$
satisfying the Leibniz rule,
\[
[e_1, fe_2] = f [e_1, e_2] + \rho(e_1) f \cdot e_2,
\]
where $e_i \in \Gamma(E)$ and $f \in C^{\infty}(M)$.
\end{Definition}
A Lie algebroid is a generalization of a Lie algebra and the space of vector fields on a smooth manifold.
\begin{Example}[Lie algebras]
Let a manifold $M$ be one point $M = \{pt \}$.
Then a Lie algebroid is a Lie algebra $\mathfrak{g}$.
\end{Example}
\begin{Example}[tangent Lie algebroids]\label{tangentLA}
If a vector bundle $E$ is a tangent bundle $TM$ and $\rho = \mathrm{id}$,
then a bracket $[-,-]$ is a normal Lie bracket
on the space of vector fields $\mathfrak{X}(M)$
and $(TM, \mathrm{id}, [-,-])$ is a Lie algebroid.
It is called a \textit{tangent Lie algebroid}.
\end{Example}

\begin{Example}[action Lie algebroids]\label{actionLA}
Suppose that there is a smooth action of a Lie group~$G$
to a smooth manifold $M$,
$M \times G \rightarrow M$.
The differential of the map induces an infinitesimal action of the Lie algebra $\mathfrak{g}$ of $G$ on a manifold $M$.
Since $\mathfrak{g}$ acts as a differential operator on $M$,
the differential of the map determines a bundle map $\rho\colon M \times \mathfrak{g} \rightarrow TM$.
Consistency of a Lie bracket requires that $\rho$ is
a Lie algebra morphism such that
\begin{gather}
[\rho(e_1), \rho(e_2)] = \rho([e_1, e_2]),\label{almostLA}
\end{gather}
where the bracket in left-hand side
of \eqref{almostLA} is a Lie bracket of vector fields.
These data gives a~Lie algebroid $(E= M \times \mathfrak{g}, \rho, [-,-])$.
This Lie algebroid is called an \textit{action Lie algebroid}.
\end{Example}

\begin{Example}[Poisson Lie algebroids]\label{Poisson}
A bivector field $\pi \in \Gamma\bigl(\wedge^2 TM\bigr)$ is called a Poisson bivector field if $[\pi, \pi]_S =0$, where $[-,-]_S$ is a Schouten bracket on the space of multivector fields, $\Gamma(\wedge^{\bullet} TM)$.
If a smooth manifold $M$ has a Poisson bivector field $\pi$, $(M, \pi)$ is called a Poisson manifold.
For a Poisson bivector field $\pi$, a bundle map is defined as
$\rho= -\pi^{\sharp}\colon T^*M \rightarrow TM$ by~${\bracket{\pi^{\sharp}(\alpha)}{\beta}
= \pi(\alpha, \beta)}$ for all $\beta \in \Omega^1(M)$.

Let $(M, \pi)$ be a Poisson manifold. Then, a Lie algebroid structure is introduced on $T^*M$.
$\pi^{\sharp}\colon T^*M \rightarrow TM$ is the anchor map, and
a Lie bracket on $\Omega^1(M)$ is defined by the so called Koszul bracket,
\begin{align}
[\alpha, \beta]_{\pi} = \calL_{\pi^{\sharp} (\alpha)}\beta - \calL_{\pi^{\sharp} (\beta)} \alpha - \rd(\pi(\alpha, \beta)),
\label{Koszulbracket}
\end{align}
where $\alpha, \beta \in \Omega^1(M)$.
\end{Example}

\begin{Example}[Lie algebroids induced from twisted Poisson structures]\label{twistedP}
If a bivector field $\pi \in \Gamma\bigl(\wedge^2 TM\bigr)$ and
$H \in \Omega^3(M)$ satisfy
\begin{gather}\label{tPoisson1}
 \frac{1}{2}[\pi, \pi]_S
= \bigl\langle\otimes^{3} \pi,H\bigr\rangle,
\qquad \rd H =0,
\end{gather}
$(M, \pi, H)$ is called a twisted Poisson manifold \cite{Ikeda:2019czt,Klimcik:2001vg, Park:2000au, Severa:2001qm}.
Here $\bigl\langle\otimes^{3} \pi,H\bigr\rangle$ is defined by
$\bigl\langle\otimes^{3} \pi,H\bigr\rangle(\alpha_1, \alpha_2, \alpha_3)
:= H\bigl(\pi^{\sharp} (\alpha_1), \pi^{\sharp} (\alpha_2), \pi^{\sharp} (\alpha_3)\bigr)
$ for $\alpha_i \in \Omega^1(M)$.

If we take the same bundle map, $-\pi^{\sharp}\colon T^*M \rightarrow TM$
as in Example \ref{Poisson},
and a Lie bracket deformed by $H$,
\[
[\alpha, \beta]_{\pi,H} = \calL_{\pi^{\sharp} (\alpha)}\beta - \calL_{\pi^{\sharp} (\beta)} \alpha - \rd(\pi(\alpha, \beta))
+ \iota_{\pi^{\sharp}(\alpha)} \iota_{\pi^{\sharp}(\beta)} H
\]
for $\alpha, \beta \in \Omega^1(M)$.
Then, $\bigl(T^*M, -\pi^{\sharp}, [-, -]_{\pi, H}\bigr)$ is a Lie algebroid.
\end{Example}
One can refer to proper documents for basic properties of Lie algebroids can be found, for instance, in \cite{Mackenzie}.

\subsection{Lie algebroid differentials}\label{LAdiffandconn}
For a Lie algebroid $E$, sections of the exterior algebra of $E^*$ are called
\textit{$E$-differential forms}.
Consider the spaces of $E$-differential forms $\Gamma(\wedge^{\bullet} E^*)$.
A differential ${}^E \rd$ on $\Gamma(\wedge^{\bullet} E^*)$ is
called a~\textit{Lie algebroid differential} (or simply, a differential)
${}^E \rd\colon \Gamma(\wedge^m E^*)
\rightarrow \Gamma\bigl(\wedge^{m+1} E^*\bigr)$ and
is defined as follows.
\begin{Definition}
A Lie algebroid differential ${}^E \rd\colon \Gamma(\wedge^m E^*)
\rightarrow \Gamma\bigl(\wedge^{m+1} E^*\bigr)$ is defined by
\begin{align*}
{}^E \rd \alpha(e_1, \dots, e_{m+1})
={}&\sum_{i=1}^{m+1} (-1)^{i-1} \rho(e_i) \alpha(e_1, \dots,
\check{e_i}, \dots, e_{m+1})
 \\ &
+ \sum_{1 \leq i < j \leq m+1} (-1)^{i+j} \alpha([e_i, e_j], e_1, \dots, \check{e_i}, \dots, \check{e_j}, \dots, e_{m+1}),
\end{align*}
where $\alpha \in \Gamma(\wedge^m E^*)$ and $e_i \in \Gamma(E)$.
\end{Definition}
$\bigl({}^E \rd\bigr)^2=0$ is satisfied.
The Lie algebroid differential ${}^E \rd$ is a generalization of the de Rham differential on $T^*M$ and the Chevalley--Eilenberg differential on a Lie algebra.

\section[Compatible E-n-forms]{Compatible $\boldsymbol{E}$-$\boldsymbol{n}$-forms}\label{sec:}
\subsection{Definition and consistency}

In this section, we define a geometric structure on a Lie algebroid
over a pre-multisymplectic manifold.
First, we introduce a pre-multisymplectic manifold.
\begin{Definition}
A closed $(n+1)$-form $\omega \in \Omega^{n+1}(M)$ is called
a pre-$n$-plectic form.
$(M, \omega)$ is called a pre-$n$-plectic manifold.
\end{Definition}

If we additionally suppose that $\omega$ is nondegenerate,
$\omega$ is called an $n$-plectic form.
A $1$-plectic form is a symplectic form.

Let $\rde\colon \Gamma(\wedge^m E^*) \rightarrow \Gamma\bigl(\wedge^{m+1} E^*\bigr)$ be a Lie algebroid differential.
Given $\omega \in \Omega^{n+1}(M)$,
${\iota_{\rho}^k}$ is
defined by
\begin{align*}
\iota_{\rho}^{k} \momega(v_{k+1}, \dots, v_{n+1}) (e_1, \dots, e_k)
={}&
\iota_{\rho(e_1)} \dots \iota_{\rho(e_k)} \momega(v_{k+1}, \dots, v_{n+1})
 \\
={}& \bigl\langle\otimes^{k}\rho,\momega\bigr\rangle
(e_k, \dots, e_1, v_{k+1}, \dots, v_{n+1})
 \\
:={}&
\momega(\rho(e_k), \dots, \rho(e_1), v_{k+1}, \dots, v_{n+1})
\end{align*}
for $e_1, \dots, e_k \in \Gamma(E)$ and
$v_{k+1}, \dots, v_{n+1} \in \mathfrak{X}(M)$.

\begin{Definition}\label{bracomp}
Let $(M, \omega)$ be a pre-$n$-plectic manifold
and $(E, \rho, [-,-])$ be a Lie algebroid over $M$.
Then, an $E$-$n$-form $J \in \Gamma(\wedge^n E^*)$ is called
\textit{compatible} if $J$ satisfies
\begin{gather}
 {}^E \rd J
= - \iota_{\rho}^{n+1} \omega
= - \bigl\langle\otimes^{n+1} \rho,\omega\bigr\rangle.
\label{momsec1}
\end{gather}
\end{Definition}

Since $\rde^2=0$,
\begin{align}
\rde \bigl(\iota^{n+1}_{\rho} \omega\bigr) = 0,
\label{consitency}
\end{align}
must be satisfied for consistency.
The left-hand side is
\begin{align*}
& \rde \bigl(\iota^{n+1}_{\rho} \omega\bigr)(e_{n+2}, \dots, e_1) =
\rd \omega(\rho(e_1), \dots, \rho(e_{n+2}))\\ &\phantom{\qquad=}{}
+ \sum_{1 \leq i < j \leq n+2} (-1)^{i+j-1}
\omega([\rho(e_i), \rho(e_j)], \rho(e_1), \dots,
\check{\rho(e_i)}, \dots, \check{\rho(e_j)}, \dots, \rho(e_{n+2})) \\ &\phantom{\qquad=}{}
+ \sum_{1 \leq i < j \leq n+2} (-1)^{i+j} \omega(\rho([e_i, e_j]), \rho(e_1), \dots, \check{\rho(e_i)}, \dots, \check{\rho(e_j)}, \dots, \rho(e_{n+2}))
\end{align*}
for $e_i \in \Gamma(E)$.
Since $[\rho(e_1), \rho(e_2)] = \rho([e_1, e_2])$ is satisfied in a Lie algebroid $E$, we obtain the following proposition.
\begin{Proposition}\label{consitency1}
If $\omega$ is closed, \eqref{consitency} is satisfied and equation \eqref{momsec1} is consistent with $\rho$, $[-,-]$ and $\omega$.
\end{Proposition}

In other words, Definition \ref{bracomp} means that
if $- \iota_{\rho}^{n+1} \omega$ is ${}^E \rd$-exact,
and its $E$-$n$-form is $J$, $(J, \omega)$ is called compatible.

\subsection{Gauge transformations}
Recall that $\iota^{n+1}_{\rho} \omega$ is closed with
respect to the Lie algebroid differential $\rde$,
$\rde \bigl(\iota^{n+1}_{\rho} \omega\bigr) = 0$ from consistency with
equations \eqref{momsec1} and \eqref{consitency}.
Obvious equivalence relation of compatible $E$-$n$-forms is
difference of ${}^E \rd$-exact terms.
\begin{Definition}\label{bracomp2}
For a Lie algebroid $E$, two compatible $E$-$n$-forms $J$ and $J'$ are
equivalent if there exists an $E$-$(n-1)$-form $K \in \Gamma\bigl(\wedge^{n-1} E^*\bigr)$
 satisfying $J^{\prime} = J + {}^E \rd K$.
\end{Definition}
In fact, if $J$ is a compatible $E$-$n$-form $J^{\prime}$ also satisfies the same compatible condition.

We consider another nontrivial equivalence.
For a pre-$n$-plectic form $\omega$, we consider a deformation
$\omega^{\prime} = \omega + \rd \lambda$ for an arbitrary $n$ form $\lambda \in \Omega^{n}(M)$, which is called a \textit{gauge transformation}.
Obviously, if $\omega$ is a pre-$n$-plectic form, $\omega^{\prime}$ is also
a pre-$n$-plectic form.

Using the identity $[\rho(e_1), \rho(e_2)] = \rho([e_1, e_2])$ of a Lie algebroid, the following equation holds,
\[ \rde \bigl(\iota_{\rho}^n \lambda\bigr) = \iota_{\rho}^{n+1} \rd \lambda
\]
for $\lambda \in \Omega^{n}(M)$.
Thus, if $J$ is a compatible $E$-$n$-form for $\omega$,
and $\omega^{\prime} = \omega -\rd \lambda$,
the compatible $E$-$n$-form $J^{\prime}$ for $\omega^{\prime}$
is given by $J^{\prime} = J - \iota_{\rho}^n \lambda$.
In fact,
\begin{gather*}
\rde J^{\prime}
= - \iota_{\rho}^{n+1} \omega^{\prime}
= - \iota_{\rho}^{n+1} \omega - (\iota_{\rho})^{n+1} \rd \lambda
= - \iota_{\rho}^{n+1} \omega - \rde \bigl(\iota_{\rho}^n \lambda\bigr)=\rde\bigl(J - \iota_{\rho}^n \lambda\bigr).
\end{gather*}
A gauge equivalence of two compatible $E$-$n$-forms is formulated as follows.
\begin{Definition}\label{equivalentEnform}
Two compatible $E$-$n$-forms $J$ and $J^{\prime}$ are \textit{gauge equivalent} if there exists an $n$-form $\lambda \in \Omega^{n}(M)$ such that
$J^{\prime} = J - \iota_{\rho}^n \lambda$.
\end{Definition}
In Section~\ref{sec:Exam}, we discuss gauge equivalence of
the twisted Poisson structure in Example~\ref{tPoissonExample}.

\section{Examples}\label{sec:Exam}

The compatible $E$-$n$-form \eqref{momsec1}
is considered as a universal compatibility condition
of a Lie algebroid structure with a pre-multisymplectic form.
It appears in various geometries. In this section, we list up examples.

\begin{Example}[Liouville $1$-forms]
Take $n=1$ and let $(M, \omega)$ be a symplectic manifold and $E=TM$ be a tangent bundle.
$TM$ is a tangent Lie algebroid with the anchor map $\rho = \mathrm{id}$
(see Example \ref{tangentLA}).
An $E$-$1$-form $J$ is a normal differential form on $M$,
and a Lie algebroid differential ${}^E \rd = \rd_{\mathrm{dR}}$
is the de Rham differential.
The compatible condition \eqref{momsec1} is
\[ \rd_{\mathrm{dR}} J = - \omega,
\]
which means that $J$ is the Liouville $1$-form for the symplectic form $\omega$.\end{Example}

\begin{Example}[Poisson and twisted Poisson structures]\label{tPoissonExample}
Let $(\pi, H)$ be a twisted Poisson structure on $M$ (see Example \ref{twistedP}).
Note that it reduces to a Poisson structure if $H=0$.
The cotangent bundle $T^*M$ has a Lie algebroid structure
as explained in Examples \ref{Poisson} and \ref{twistedP}.
Formulas of the twisted Poisson structures and induced Lie algebroids are in
\cite{Kosmann-Schwarzbach:2005evu,Severa:2001qm}.
Equation~\eqref{tPoisson1} is equivalent to ${}^E \rd \pi = - \bigl\langle\otimes^{3} \pi,H\bigr\rangle$,
where ${}^E \rd$ is the Lie algebroid differential of this Lie algebroid
$T^*M$.
Thus, $J= \pi$ is a compatible $TM$-$2$-form on a pre-$2$-plectic manifold
with a~pre-$2$-plectic form $H$.

Next let us analyze gauge equivalence.
A gauge transformation of the twisted Poisson structure is
$H^{\prime} = H + \rd B$ with a $2$-form $B$.
Then, the compatible $E$-$2$-form with respect to $H^{\prime}$ is~$\pi^{\prime} = \pi - \bigl\langle\otimes^2 \pi,B\bigr\rangle$
from Definition \ref{equivalentEnform}.

Note that gauge transformations in Definition \ref{equivalentEnform} are
different from the definition of the gauge transformation in the twisted Poisson structure \cite{BursztynRadko,Severa:2001qm}.
In the standard definition, the bivector field changes to
$\pi^{\prime} = \pi (1 + \pi B)^{-1}$ under $H^{\prime} = H + \rd B$.

In the standard definition, equation \eqref{tPoisson1} gives the Lie algebroid structure on $T^*M$ and a~compatible $E$-$n$-form simultaneously.
If the gauge transformation $\pi^{\prime} = \pi (1 + \pi B)^{-1}$ consistently
deforms the Lie algebroid structure defined by $(\pi^{\prime}, H)$ and a compatible $E$-$n$-form to~${J^{\prime} = \pi^{\prime}}$ simultaneously.

However, our gauge transformation $\pi^{\prime} = \pi - \bigl\langle\otimes^2 \pi,B\bigr\rangle$ only deforms the compatible $E$-$n$-form to $J^{\prime} = \pi^{\prime}$
under the fixed Lie algebroid constructed from $(\pi, H)$.
The reason is for generalizations to higher pre-$n$-plectic cases. In a pre-$n$-plectic case, the simultaneous deformation of the Lie algebroid and the compatible $E$-$n$-form like the twisted Poisson structure does not work.
The twisted Poisson structure is a very special case and
our gauge transformation is applicable to general $E$-$n$-form on a higher $n$-plectic manifold.
\end{Example}

\begin{Example} 
A twisted $R$-Poisson structure $(\pi, H, J)$ on a~smooth manifold~$M$
is defined as follows \cite{Chatzistavrakidis:2021nom}.\footnote{Since $J$ is denoted by $R$ in the paper \cite{Chatzistavrakidis:2021nom},
it is called the $R$-Poisson structure.}
$M$ is a smooth manifold and $\pi \in \Gamma\bigl(\wedge^2 TM\bigr)$ is a Poisson bivector field.
 $H$ is a~closed $(n+1)$-form, and $J \in \Gamma(\wedge^n TM)$ is an $n$-multivector field.
As explained in Example~\ref{Poisson},
the Poisson bivector field $\pi$ induces a Lie algebroid structure
on $T^*M$, where $\rho = - \pi^{\sharp}$ and the Lie bracket is given by~\eqref{Koszulbracket}.
Under this Lie algebroid structure,
$(H, J)$ is the twisted $R$-Poisson structure if it satisfies
\[
{}^E \rd J = (- 1)^n \bigl\langle\otimes^{n+1} \pi,H\bigr\rangle.
\]
$J$ is a compatible $T^*M$-$n$-form with the pre-$(n+1)$-plectic form $\omega = (- 1)^{n+1} H$.
\end{Example}

\begin{Example}[momentum maps]
We assume a smooth action of a Lie group $G$ on a symplectic manifold $(M, \omega)$. Let $\mathfrak{g}$ be a Lie algebra for $G$.
The Lie group action induces the Lie algebra action
$\rho\colon\mathfrak{g} \rightarrow TM$.
A map $\mu\colon M \rightarrow \mathfrak{g}^*$ is called a momentum map if it satisfies
\begin{gather}
\rd \mu(e) = - \iota_{\rho(e)} \omega,\label{mommap01}
\\
\rho(e_1) \bracket{\mu}{e_2} = \bracket{\mu}{[e_1, e_2]}\label{mommap02}
\end{gather}
for $e, e_1, e_2 \in \mathfrak{g}$.
Here $\bracket{-}{-}$ is the pairing of $\mathfrak{g}^*$ and $\mathfrak{g}$.
As explained in Example \ref{actionLA}, the Lie group action induces
an action Lie algebroid structure on a trivial bundle
$E = M \times \mathfrak{g}$.
The anchor map is induced from the map of the Lie algebra action
as $\rho\colon M \times \mathfrak{g} \rightarrow TM$.
The momentum map is regarded as a section of the dual trivial bundle
$\mu \in \Gamma(M \times \mathfrak{g}^*)$.
Under the condition \eqref{mommap01}, equation \eqref{mommap02} is changed to
\cite{Blohmann:2018}
\[
{}^E \rd \mu(e_1, e_2) = - \iota_{\rho}^2 \omega(e_1, e_2),
\]
i.e., ${}^E \rd \mu = - \iota_{\rho}^2 \omega$.
Here ${}^E \rd$ is a Lie algebroid differential with respect to the action Lie algebroid. Thus we obtain the following result.
\begin{Proposition}
A momentum map $\mu\colon M \rightarrow \mathfrak{g}^*$ is a compatible $E$-$1$-form for the action Lie algebroid induced from a Lie group action on $M$.
\end{Proposition}
\end{Example}

\begin{Example}[homotopy moment(um) maps]\label{hmm}\cite{Fregier:2013dda}
Let $(M, \momega)$ be an $n$-plectic manifold.
Assume an action of a Lie group $G$ on $M$ as in the case of the momentum map.
The action of $G$ induces the corresponding infinitesimal Lie algebra action on $M$ as vector fields, $\rho\colon\mathfrak{g} \rightarrow TM$.
We consider the space,
\smash{$\Omega^{k}\bigl(M, \wedge^{n-k} \mathfrak{g}^*\bigr) =
\Omega^{k}(M) \otimes \wedge^{n-k} \mathfrak{g}^*$},
and two differentials~${\rd := \rd \otimes 1}$ and~${\rd_{\mathrm{CE}} := 1 \otimes \rd_{\mathrm{CE}}}$
where $\rd$ is the de Rham differential on $\Omega^{k}(M)$ and
$\rd_{CE}$ is the Chevalley--Eilenberg differential on $\wedge^{n-k} \mathfrak{g}^*$.
Introduce a $k$-form taking a value in $\wedge^{n-k} \mathfrak{g}^*$,
\smash{$\hmu_k \in \Omega^{k}\bigl(M, \wedge^{n-k} \mathfrak{g}^*\bigr)$},
where $k=0, \dots, n-1$,
and their formal sum, $\hmu = \sum_{k=0}^{n-1} \hmu_k$.
\begin{Definition}
A homotopy momentum map $\hmu$ is defined by\footnote{Sign factors of the equation are different from \cite{Fregier:2013dda}. Sign factors are arranged to original ones by proper redefinitions.}
\begin{gather}
 (\rd + \rd_{CE}) \hmu = \sum_{k=0}^{n} (-1)^{n-k+1} \iota_{\rho}^{n+1-k} \momega.
\label{homotopyMM}
\end{gather}
\end{Definition}
A momentum map on an $n$-plectic manifold \cite{Carinena:1992rb, Gotay:1997eg}
is a homotopy momentum map.
Expanded by the form degree,
equation \eqref{homotopyMM} becomes the following $n$ equations,
\begin{gather}
 \rd \hmu_{n-1} = - \iota_{\rho}^1 \momega,\nonumber
\\
 \rd \hmu_{n-2} + \rd_{CE} \hmu_{n-1}= \iota_{\rho}^{2} \momega,
\nonumber
\\
 \vdots
\nonumber
\\
 \rd \hmu_{k-1} + \rd_{CE} \hmu_{k}= (-1)^{n-k+1} \iota_{\rho}^{n+1-k} \momega,
\nonumber
\\
 \vdots
\nonumber
\\
 \rd_{CE} \hmu_{0} = (-1)^{n+1} \iota_{\rho}^{n+1} \momega.
\label{homotopyMM4}
\end{gather}

The Lie group action induces an action Lie algebroid structure on
the trivial bundle ${E \!= \!M \!\times \!\mathfrak{g}}$.
The Lie algebroid differential on the action Lie algebroid is
${}^E \rd = \operatorname{ad}^*_{\rho} + \rd_{CE}$,
where $\operatorname{ad}^*$ is the following action on $\alpha$ induced from the Lie algebra action on $M$,
\[
\operatorname{ad}^*_{\rho} \alpha(e_1, \dots, e_{m+1})
:=
\sum_{i=1}^{m+1} (-1)^{i-1} \rho(e_i) \alpha(e_1, \dots,
\check{e_i}, \dots, e_{m+1})
\]
for $e_i \in \mathfrak{g}$,
We can rewrite $n$ equations
\eqref{homotopyMM4}
to equivalent $n$ equations \cite{Hirota:2021isx},
\begin{gather}
 \rd \mu_{n-1} = - \iota_{\rho}^1 \momega,\nonumber
\\
 \rd \mu_{n-2} + {}^E \rd \mu_{n-1}= - \iota_{\rho}^{2} \momega,\nonumber
\\
 \vdots
\nonumber \\
 \rd \mu_{k-1} + {}^E \rd \mu_{k}= - \iota_{\rho}^{n+1-k} \momega,\nonumber
\\
 \vdots
\nonumber \\
 {}^E \rd \mu_{0} = - \iota_{\rho}^{n+1} \momega,
\label{homotopyMS24}
\end{gather}
where $\hmu_k = (-1)^{n-k+1} \mu_k$.
The last equation \eqref{homotopyMS24} shows that $J=\mu_0$ is
a compatible $E$-$n$-form.
\begin{Proposition}
$\hmu_0 = (-1)^{n+1} \mu_0$ in a homotopy momentum map is a compatible $E$-$n$-form.
\end{Proposition}
\end{Example}

\begin{Example}[momentum sections]
A momentum map is generalized to a momentum section~\cite{Blohmann:2018}.
See also \cite{Ikeda:2019pef, Ikeda:2021fjk,Kotov:2016lpx}.
A momentum section is also an example of a compatible $E$-$n$-form.

Suppose that a base manifold $M$ is a pre-symplectic manifold with a pre-symplectic form~$\omega$.
Moreover, suppose a Lie algebroid $(E, \rho, [-,-])$ over $M$.
In order to define a momentum section,
we introduce a connection on the vector bundle $E$,
$\nabla\colon\Gamma(E)\rightarrow \Gamma(T^*M \otimes E)$, which is a
$\bR$-linear map satisfying the Leibniz rule,
\[
\nabla (f e) = f \nabla e + (\rd f) \otimes e
\]
for $e \in \Gamma(E)$ and $f \in C^{\infty}(M)$.
A dual connection on $E^*$ is defined by the equation,
\[
\rd \bracket{\mu}{e} = \bracket{\nabla \mu}{e} + \bracket{\mu}{\nabla e}
\]
for all sections $\mu \in \Gamma(E^*)$ and $e \in \Gamma(E)$.

A momentum section is defined as follows.
\begin{Definition}
A section $\mu \in \Gamma(E^*)$ of $E^*$ is called a bracket-compatible
momentum section if $\mu$ satisfies the following two conditions:\footnote{One more condition $\nabla^2 \mu =0$ is imposed in~\cite{Blohmann:2018}. The condition is dropped in this paper, since it is not necessary in our context.}
\begin{gather}
\nabla \mu = - \iota_{\rho} \omega,
\qquad {}^E \rd \mu
= - \iota_{\rho}^{2}{\omega}.\label{HH3}
\end{gather}
\end{Definition}

The second equation \eqref{HH3} is nothing but the condition
\eqref{momsec1} of a compatible $E$-$1$-form for~${J=\mu}$.
\begin{Proposition}
A momentum section $\mu$ is a compatible $E$-$1$-form.
\end{Proposition}
\end{Example}

\begin{Example}[homotopy momentum sections]\cite{Hirota:2021isx}
A homotopy momentum map on a pre-multisymplectic manifold
is generalized to a homotopy momentum section.

As the preparation to introduce a homotopy momentum section, we define another derivation called an $E$-connection on a Lie algebroid.
\begin{Definition}
Let $E$ be a Lie algebroid over a smooth manifold $M$ and $E'$ be a vector bundle over the same base manifold $M$.
An \textit{$E$-connection} on a vector bundle $E'$
with respect to the Lie algebroid $E$ is a $\bR$-linear
map
${}^E \nabla\colon \Gamma(E') \rightarrow \Gamma(E^* \otimes E')$
satisfying
\[
{}^E \nabla_e (f e') = f {}^E \nabla_e e' + (\rho(e) f) e'
\]
for $e \in \Gamma(E)$, $e' \in \Gamma(E')$ and $f \in C^{\infty}(M)$.
\end{Definition}
The ordinary connection is regarded as an
$E$-connection for $E=TM$, $\nabla = {}^{TM} \nabla$.

The \textit{standard $E$-connection} on $E'=E$,
${}^E \nabla\colon \Gamma(E) \rightarrow \Gamma(E^* \otimes E)$ is defined by
\[
{}^E \nabla_{e} e^{\prime} := \nabla_{\rho(e)} {e^{\prime}}
+ [e, e^{\prime}]
\]
for $e, e^{\prime} \in \Gamma(E)$.
A general $E$ connection ${}^E \nabla$ on $E$ is given by
\[ {}^E \nabla_{e} e^{\prime} = \nabla_{\rho(e)} {e^{\prime}}
+ [e, e^{\prime}]
- \chi(e, e^{\prime})
\]
with some tensor $\chi \in \Gamma(E \otimes E^* \otimes E^*)$.

If a normal connection $\nabla$ on $E$ as a vector bundle is given,
we can define the following canonical
$E$-connection for the tangent bundle $E'=TM$.
An (canonical) $E$-connection on a~tangent bundle, ${}^E \nabla\colon \Gamma(TM) \rightarrow \Gamma(E^* \otimes TM)$ is defined by
\[
{}^E \nabla_{e} v := \calL_{\rho(e)} v + \rho(\nabla_v e)
= [\rho(e), v] + \rho(\nabla_v e),
\]
where
$e \in \Gamma(E)$ and $v \in \mathfrak{X}(M)$.
It is also called the opposite connection.
One can refer to~\cite{AbadCrainic, CrainicFernandes, DufourZung}
about the general theory of connections on a Lie algebroid.

Given a normal connection $\nabla$, a covariant derivative is generalized
to the derivation on the space of differential forms,
$\rd^{\nabla}\colon \Omega^k(M, \wedge^m E^*)
\rightarrow \Omega^{k+1}(M, \wedge^m E^*)$,
called an \textit{exterior covariant derivative}.\footnote{The exterior covariant derivative is also denoted by $\nabla$ in \cite{Hirota:2021isx}.}
Similarly, given an $E$-connection ${}^E \nabla$, the Lie algebroid differential ${}^E \rd$ can be generalized to the operation on the space
\begin{gather*}
\Gamma\bigl(\wedge^k E' \otimes \wedge^m E^*\bigr), \qquad
{}^E \rd^{\nabla}\colon
\Gamma\bigl(\wedge^k E' \otimes \wedge^m E^*\bigr)
\rightarrow \Gamma\bigl(\wedge^k E' \otimes \wedge^{m+1} E^*\bigr).
\end{gather*}
 For instance, see \cite{AbadCrainic}.
The concrete equation for $E'=T^*M$ is as follows.
\begin{Definition}
For \[\Omega^k(M, \wedge^m E^*) = \Gamma\bigl(\wedge^k T^* M \otimes \wedge^m E^*\bigr),\]
the \textit{$E$-exterior covariant derivative}
${}^E \rd^{\nabla}\colon \Omega^k(M, \wedge^m E^*) \rightarrow \Omega^k\bigl(M, \wedge^{m+1} E^*\bigr)$ is defined by
\begin{align*}
{}^E \rd^{\nabla} \alpha(e_1, \dots, e_{m+1})
:={}& \sum_{i=1}^{m+1} (-1)^{i-1}
{}^E \nabla_{e_i}
(\alpha(e_1, \dots,
\check{e_i}, \dots, e_{m+1})) \\ &
+ \sum_{1 \leq i < j \leq m+1} (-1)^{i+j} \alpha([e_i, e_j], e_1, \dots, \check{e_i}, \dots, \check{e_j}, \dots, e_{m+1})
\end{align*}
for $\alpha \in \Omega^k(M, \wedge^m E^*)$ and $e_i \in \Gamma(E)$.
\end{Definition}

Let $\mu_k \in \Omega^k\bigl(M, \wedge^{n-k} E^*\bigr)
$ be a $k$-form taking values in $\wedge^{n-k} E^*$, where $k=0,\dots, n-1$.
Both the exterior covariant derivative $\nabla$
and the $E$-exterior covariant derivative ${}^E \rd^{\nabla}$
act on the space~${\Omega^k\bigl(M, \wedge^{n-k} E^*\bigr)}$.

\begin{Definition}\label{defhms}
A formal sum $\mu = \sum_{k=0}^{n-1} \mu_k$ is called a \textit{homotopy momentum section} if $\mu$ satisfies
\begin{gather}
 \bigl(\rd^{\nabla} + {}^E \rd^{\nabla}\bigr) \mu = - \sum_{k=0}^{n} \iota_{\rho}^{n+1-k} \momega.
\label{homotopyMS}
\end{gather}
\end{Definition}
Expanding the equation by form degree of both sides,
equation \eqref{homotopyMS} is the following $n$ equations,
\begin{gather}
 \rd^{\nabla} \mu_{n-1} = - \iota_{\rho}^1 \momega,\nonumber
\\
 \rd^{\nabla} \mu_{n-2} + {}^E \rd^{\nabla} \mu_{n-1}= - \iota_{\rho}^{2} \momega,\nonumber
\\
 \vdots
\nonumber \\
 \rd^{\nabla} \mu_{k-1} + {}^E \rd^{\nabla} \mu_{k}= - \iota_{\rho}^{n+1-k} \momega,\nonumber
\\
 \vdots
\nonumber \\
 {}^E \rd^{\nabla} \mu_{0} = {}^E \rd \mu_{0} = - \iota_{\rho}^{n+1} \momega.
\label{homotopyMS4}
\end{gather}
The final equation \eqref{homotopyMS4} is the condition that
$J= \mu_0$ is a compatible $E$-$n$-form \eqref{momsec1}.
\begin{Proposition}
$\mu_0$ in a homotopy momentum section is a compatible $E$-$n$-form.
\end{Proposition}
\end{Example}

\section[Higher Dirac structures on E oplus wedge\^\{n-1\} E\^*]{Higher Dirac structures on $\boldsymbol{E \oplus \wedge^{n-1} E^*}$}\label{HDirac}
In this section, we reformulate a compatible $E$-$n$-form in terms of
a higher Dirac structure of a~Lie $n$-algebroid.
We consider a Lie $n$-algebroid on $E \oplus \wedge^{n-1} E^*$
induced from a Lie algebroid~$E$.

Let $(E, \rho, [-,-])$ be a Lie algebroid over $M$. Then, an almost Lie algebroid structure is induced on $E \oplus \wedge^{n-1} E^*$.
Three operations are constructed from the Lie algebroid operations. They are
a bilinear symmetric form $\inner{-}{-}\colon \Gamma\bigl(E \oplus \wedge^{n-1} E^*\bigr) \times \Gamma\bigl(E \oplus \wedge^{n-1} E^*\bigr) \rightarrow \Gamma\bigl(\wedge^{n-2} E^*\bigr)$,
a~bundle map $\hrho\colon E \oplus \wedge^{n-1} E^* \rightarrow TM$
called the anchor map,
and a bilinear bracket $\courant{-}{-}\colon \Gamma\bigl(E \oplus \wedge^{n-1} E^*\bigr)
\times \Gamma\bigl(E \oplus \wedge^{n-1} E^*\bigr) \rightarrow \Gamma\bigl(E \oplus \wedge^{n-1} E^*\bigr)$ called a (higher) Dorfman bracket.
Let $u + \alpha, v + \beta \in \Gamma\bigl(E \oplus \wedge^{n-1} E^*\bigr)$,
where $u, v \in \Gamma(E)$ and
$\alpha, \beta \in \Gamma\bigl(\wedge^{n-1} E^*\bigr)$.
Then three operations are defined by
\begin{align}
&\inner{u + \alpha}{v + \beta} = \iota_u \beta + \iota_v \alpha,
\qquad\hrho(e)f = \rho(u) f,\label{Lien01}
\\
&\courant{u + \alpha }{v + \beta} = [u, v] + {}^E \calL_{u} \beta - \iota_v {}^E \rd \alpha + \iota_u \iota_v \bigl(\iota_{\rho}^{n+1} \omega\bigr),
\label{Lien03}
\end{align}
where the interior product $\iota_v$ is the contraction of an element $v \in \Gamma(E)$ with $\Gamma(E^*)$,
$\rho(u)$ is the anchor map of $E$,
and ${}^E \calL_u := \iota_u {}^E \rd + {}^E \rd \iota_u$ is
the $E$-Lie derivative.
The Dorfman bracket \eqref{Lien03} is not skew-symmetric but satisfies the Jacobi type identity,
\begin{align}
&\courant{\courant{u + \alpha }{v + \beta}}{w + \gamma}
+ \courant{\courant{v + \beta}{w + \gamma}}{u + \alpha }\nonumber\\
&\qquad+ \courant{\courant{w + \gamma}{u + \alpha }}{v + \beta}
=0.
\label{Leibniz}
\end{align}
An algebroid satisfying equation \eqref{Leibniz} is called a \textit{Leibniz algebroid}.
Especially, we call
\[\bigl(E \oplus \wedge^{n-1} E^*, \inner{-}{-}, \rho, \courant{-}{-}\bigr)\]
a \textit{Lie $n$-algebroid}, or a \textit{Vinogradov algebroid}
induced from a Lie algebroid~$E$ \cite{Grutzmann:2014hkn}.
We only consider this specific Lie $n$-algebroid in this paper.
There are many analysis about general theories for Lie $n$-algebroids.
In general, Lie $n$-algebroids are induced from a QP-manifold of degree~$n$.
A general concept of Lie $n$-algebroids is referred to, for instance,
\cite{Bonavolonta-Poncin, Cueca:2019rmv, Cueca:2023oxw, Hagiwara, Ikeda:2013wh, Severa:2005, Sheng-Zhu, Zambon:2010ka}.

Equations \eqref{Lien01}--\eqref{Lien03} does not contain $J$.
The compatible condition \eqref{momsec1},
${}^E \rd J = - \iota_{\rho}^{n+1} \omega$, is described as
an extra structure on the Lie $n$-algebroid, a higher Dirac structure\cite{BiSheng,Bursztyn:2016hln, Zambon:2010ka}.

\begin{Definition}
A higher Dirac structure is the subbundle $L$ of the Lie $n$-algebroid
satisfying the conditions,
$\inner{e_1}{e_2} = 0$ for all $e_1, e_2 \in \Gamma(L)$, and
$\courant{e_1}{e_2}$ is an element of $\Gamma(L)$,
i.e., $\Gamma(L)$ is involutive with respect to the bracket,
$\courant{\Gamma(L)}{\Gamma(L)} \subset \Gamma(L)$.
\end{Definition}
Note that we do not impose that $L$ is the maximal rank with respect to the inner product~${\inner{-}{-}}$ but only isotropic, which is
different from the definition of a Dirac structure of a~Courant algebroid.

For an $E$-$n$-form $J \in \Gamma(\wedge^n E^*)$, we consider the following space of sections of $E \oplus \wedge^{n-1} E^*$,
\begin{align*}
\Gamma(L) &= \big\{u + \inner{J}{u} \in \Gamma\bigl(E \oplus \wedge^{n-1} E^*\bigr)
\mid u \in \Gamma(E) \big\}.
\end{align*}
\begin{Theorem}\label{admDirac}
If $J \in \Gamma(\wedge^n E^*)$ and $\omega \in \Omega^{n+1}(M)$ satisfy ${}^E \rd J = - \iota_{\rho}^{n+1} \omega$, then
$L$ is a higher Dirac structure of a Lie $n$-algebroid
$E \oplus \wedge^{n-1} E^*$.
\end{Theorem}
\begin{proof}
The inner product of two elements of $\Gamma(L)$,
$u + \inner{J}{u}$ and $v + \inner{J}{v}$ for $u, v \in \Gamma(E)$
is
\begin{align*}
\inner{u + \inner{J}{u}}{v + \inner{J}{v}} &=
\inner{u}{{J}{v}} + \inner{{J}{u}}{v} = 0
\end{align*}
from completely skewsymmetricity of $J$.

Moreover,
using~\eqref{Lien03},
the Dorfman bracket is computed as
\begin{align}
\courant{u + \inner{J}{u}}{v + \inner{J}{v}}
& = \courant{u + \iota_u J}{v + \iota_v J}
\nonumber \\
& = [u, v] + {}^E \calL_{u} \iota_v J - \iota_v {}^E \rd \iota_u J
+ \iota_u \iota_v \bigl(\iota_{\rho}^{n+1} \omega\bigr)
\nonumber \\
& = [u, v] + {}^E \calL_{u} \iota_v J - \iota_v {}^E \calL_{u} J
- \iota_v \iota_u {}^E \rd J
+ \iota_u \iota_v \bigl(\iota_{\rho}^{n+1} \omega\bigr).
\label{Dirac01}
\end{align}
Using the formula
${}^E \calL_{u} \iota_v - \iota_v {}^E \calL_{u} = \iota_{[u, v]}$
analogous to the Cartan formula, \eqref{Dirac01} becomes
\begin{align}
& = [u, v] + \inner{J}{[u, v]} + \iota_u \iota_v {}^E \rd J
+ \iota_u \iota_v \bigl(\iota_{\rho}^{n+1} \omega\bigr),
\label{Dirac02}
\end{align}
which is an element of $\Gamma(L)$ again if
the identity ${}^E \rd J = - \iota_{\rho}^{n+1} \omega$
is satisfied.
Note that we use~$\iota_u \iota_v {}^E \rd J = -
\iota_v \iota_u {}^E \rd J$.
Since then \eqref{Dirac02} becomes
\begin{align*}
\courant{u + \inner{J}{u}}{v + \inner{J}{v}}
&= [u, v] + \inner{J}{[u, v]}
\end{align*}
for every $u, v \in \Gamma(E)$.
\end{proof}

\begin{Example}
Let $(M, \pi, H)$ be a twisted Poisson manifold and
$T^*M$ is a Lie algebroid induced from the twisted Poisson structure in Example \ref{twistedP}.
We take $n=2$ and consider a Lie $2$-algebroid~${T^*M \oplus TM \simeq TM \oplus T^*M}$. For $\alpha, \beta \in \Omega^1(M)$, $u, v \in \mathfrak{X}(M)$,
and $f \in C^{\infty}(M)$
operations are defined as
\begin{align*}
&\inner{\alpha + u}{\beta + v} = \iota_u \beta + \iota_v \alpha,
\qquad
\hrho(\alpha + u)f = \pi^{\sharp}(\alpha) (f),
\\
&\courant{\alpha + u}{\beta + v} =
[\alpha, \beta]_{\pi, H} + \calL^{\pi, H}_{\alpha} v
- \iota_{\beta} \rd_{\pi, H} u + \bigl\langle\otimes^3 \pi,H\bigr\rangle(\alpha, \beta),
\end{align*}
which give a Courant algebroid with
$(\inner{-}{-}, \hrho, \courant{-}{-})$,
where
\begin{gather*}
[\alpha, \beta]_{\pi,H} = \calL_{\pi^{\sharp} (\alpha)}\beta - \calL_{\pi^{\sharp} (\beta)} \alpha - \rd(\pi(\alpha, \beta))
+ \iota_{\pi^{\sharp}(\alpha)} \iota_{\pi^{\sharp}(\beta)} H,
\qquad
\rd_{\pi} u = [\pi, u]_S,
\\
\rd_{\pi, H} u = \rd_{\pi} u + \big\langle\bigl(\otimes^2 \pi \otimes 1\bigr)H,u\big\rangle,
\qquad
L^{\pi, H}_{\alpha} v = (\iota_{\alpha} \rd_{\pi, H} +
\rd_{\pi, H} \iota_{\alpha}) v.
\end{gather*}
Let us introduce a bivector field $J \in \Gamma\bigl(\wedge^2 TM\bigr)$ and define a subbundle defined by
\begin{align*}
L &= \big\{\alpha + \inner{J}{\alpha} \in
\Gamma(T^*M \oplus TM) \mid \alpha \in \Omega^1(M) \big\}.
\end{align*}
Iff $L$ is a Dirac structure, $J$ is a compatible $T^*M$-$2$-form, i.e., a bivector field satisfying ${}^E \rd J = - \iota_{\pi}^2 H$.
This case is special. We can take $J= \pi$ since both are elements of $\Gamma\bigl(\wedge^2 TM\bigr)$, and~${{}^E \rd \pi = - \iota_{\pi}^2 H}$ is equivalent to the condition that $(\pi, H)$ is a twisted Poisson
structure.
\end{Example}

\section{Q-manifold description}\label{Qmfd}
In this section, a compatible $E$-$n$-form is formulated by graded geometry.
A Q-manifold description based on graded geometry provides a
clear method to formulate Lie $n$-algebroids.
Concretely, the compatible $E$-$n$-form is described by a Lagrangian Q-submanifold of a Q-manifold.

A graded manifold $\calM$ is a locally ringed space $(M, \calO_{\calM})$
whose structure sheaf $\calO_{M}$ is a~$\bZ$-graded commutative algebra over an ordinary smooth manifold $M$. The grading is compatible with the supermanifold grading, that is, a~variable of even degree is commutative and a~variable of odd degree is anticommutative. By definition, the structure sheaf of $M$ is locally isomorphic to $C^{\infty}(U)\otimes S^{\bullet}(V)$, where $U$ is a local chart on $M$, $V$ is a graded vector space, and~$S^{\bullet}(V)$ is a free graded commutative ring on~$V$.
$\calO_{\calM}$ is also denoted by $C^{\infty}(\calM)$.
Grading of an element~${x \in C^{\infty}(\calM)}$
is called degree and denoted by $|x|$.
For two homogeneous elements ${f, g \in C^{\infty}(\calM)}$, the product is graded commutative, $fg = (-1)^{|f||g|}gf$.
There are some recent reviews of graded manifolds in \cite{Cattaneo:2010re, Ikeda:2012pv,
Roytenberg:2006qz, Vysoky:2021wox} and references therein.

We concentrate on nonnegatively graded manifolds with $\bZ_{\geq 0}$ grading
called an N-manifold.
\begin{Definition}
If an N-manifold $\calM$ has a vector field $Q$ of degree $+1$ satisfying $Q^2=0$, it is called a \textit{Q-manifold}, or a \textit{differential graded $($dg$)$ manifold}.
\end{Definition}
This vector field $Q$ is called a homological vector field.

Let $E$ be a vector bundle over $M$.
$E$ is a Lie algebroid if and only if $E[1]$ is a Q-manifold with a homological vector field \cite{Vaintrob}.
Concrete correspondence is as follows.
Let $(x^i, q^a)$ be local coordinates on $E[1]$ of degree $(0, 1)$.
Here $i$ is the index on the base manifold $M$ and $a$ is the index of the fiber.
A vector field of degree $+1$ is locally given by
\begin{align}
Q &= \rho^i_a(x) q^a \frac{\partial}{\partial x^i}
- \frac{1}{2} C_{bc}^a(x) q^b q^c \frac{\partial}{\partial q^a},
\label{homologicalQ}
\end{align}
where $\rho^i_a(x)$ and $C^a_{bc}(x)$ are local functions.
We define a bundle map $\rho\colon E \rightarrow TM$ and the Lie bracket
$[-,-]\colon \Gamma(E) \times \Gamma(E) \rightarrow \Gamma(E)$
as $\rho(e_a) := \rho^i_a(x) \partial_i$, $[e_a, e_b] := C_{ab}^c(x) e_c$, for the basis of sections $e_a \in \Gamma(E)$.
Then, $Q^2=0$ is equivalent that the two operations $(\rho, [-,-])$ satisfy the definition of a Lie algebroid on $E$:
all the identities of the anchor map and the Lie bracket in the Lie algebroid $E$ are obtained from the $Q^2=0$ condition.

We have another description of a Lie algebroid based on a QP-manifold.
\begin{Definition}
If an N-manifold $\calM$ has a graded symplectic form $\gomega$ of degree $n$ and a vector field $Q$ of degree $+1$ satisfies $Q^2=0$, and $\calL_Q \gomega =0$,
$(\calM, \gomega, Q)$ is called a \textit{QP-manifold} of degree $n$.
\end{Definition}
In a QP-manifold, if $n \neq 0$, there exists a Hamiltonian function $\Theta$ such that $Q = \sbv{\Theta}{-}$.
$Q^2 = 0$ is equivalent to $\sbv{\Theta}{\Theta}=0$.

For a Lie algebroid, we take an $n-1$-shifted cotangent bundle
$\calM= T^*[n-1]E[1]$ with a nonnegative integer $n-1$.
Take local coordinates on $T^*[n-1]E[1]$,
$(x^i, q^a, z_i, y_a)$ of degree $(0, 1, n-1, n-2)$.
Since $T^*[n-1]E[1]$ is a (graded) cotangent bundle, it
has a canonical graded symplectic form of degree $n-1$, $\gomega = \delta x^i \wedge \delta z_i + \delta q^a \wedge \delta y_a$, where $\delta$ is the graded de Rham differential on the graded manifold.
A graded Poisson bracket for $\gomega$ is defined by
\[
\sbv{f}{g} =
(-1)^{|f|+n} \iota_{X_f} \iota_{X_g} \gomega.
\]
The Poisson brackets $\sbv{-}{-}$ satisfy
\begin{align*}
\big\{x^i, z_j \big\} &= \delta^i_j, \qquad \big\{z_j, x^i \big\} = - \delta^i_j,
\qquad
\{q^a, y_b \} = \delta^a_b, \qquad \{y_b, q^a \} = -(-1)^{n-2} \delta^a_b.
\end{align*}
We take the following degree $n$ function on $T^*[n-1]E[1]$,
\begin{align}
\Theta_{\rm Lie} &=
\bracket{\rho(q)}{z} + \frac{(-1)^{n-1}}{2} \inner{[q, q]}{y}
= (-1)^{n-1} \rho^i_a(x) z_i q^a +
\frac{(-1)^{n-1}}{2} C_{bc}^a(x) q^b q^c y_a.
\label{LAHamiltonian}
\end{align}
Here $\langle-,-\rangle$ is the pairing of $TM$ and $T^*M$
an $\inner{-}{-}$ is the pairing of $E$ and $E^*$.
${\{\Theta_{\rm Lie}, \Theta_{\rm Lie} \}=0}$ is equivalent to
$\rho(e_a) = \rho^i_a(x) \partial_i$ and $C_{bc}^a(x)$ are the anchor map and the structure function of the Lie bracket of a Lie algebroid.

We can prove that $\iota_{Q_{\rm Lie}} \gomega = - \delta \Theta_{\rm Lie}$, equivalently,
$Q_{\rm Lie} = - \{\Theta_{\rm Lie}, - \}$ in \eqref{LAHamiltonian} satisfies~${Q_{\rm Lie}^2=0}$ and $\calL_{Q_{\rm Lie}} \gomega =0$ if and only if $E$ is a Lie algebroid.
The local coordinate expression~is
\begin{align*}
\QLie ={}&
- \{\Theta_{\rm Lie}, - \}
= \rho^i_a(x) q^a \frac{\partial}{\partial x^i}
- \frac{1}{2} C_{bc}^a(x) q^b q^c \frac{\partial}{\partial q^a}
+ \bigl((-1)^n \rho^i_a z_i
- C_{ab}^c(x) q^b y_c \bigr) \frac{\partial}{\partial y_a}
 \\ &
- \left(\partial_i\rho^j_a z_j q^a
- \frac{(-1)^n }{2} \partial_i C_{bc}^a(x) q^b q^c y_a
\right) \frac{\partial}{\partial z_i}.
\end{align*}
Here $\partial_i = \frac{\partial}{\partial x^i}$.
$\QLie$ is also denoted by $\calL_Q = \QLie$ and called the tangent lift of
the vector field~$Q$ in \eqref{homologicalQ}.
$\QLie$ is the homological vector field such that $Q = \QLie|_{E[1]}$.
$T^*[n-1]E[1]$ has a~QP-manifold structure induced from a Lie algebroid structure on $E$.

We consider a function of degree $n$,
\begin{align*}
\uJ &:= \frac{1}{n!} J_{a_1 \dots a_n}(x) q^{a_1} \cdots q^{a_n},
\end{align*}
corresponding to the $E$-$n$-form $J \in \Gamma(\wedge^n E^*)$, where
\begin{align*}
\frac{1}{n!} J_{a_1 \dots a_n}(x) &:= J(e_{a_1}, \dots, e_{a_n}).
\end{align*}
The term $\uJ$ is summed to $\Theta_{\rm Lie}$ as $\Theta_{{\rm Lie}\,J} = \Theta_{\rm Lie} + (-1)^{n-1} \uJ$, and consider the vector field of degree $+1$
\begin{align*}
Q_{{\rm Lie}\,J} &= - \sbv{\Theta_{{\rm Lie}\,J}}{-}.
\end{align*}
$Q_{{\rm Lie}\,J}$ is not necessarily a homological vector field since
$Q_{{\rm Lie}\,J}^2$ is not necessarily zero without any condition for $\uJ$.
In order to realize the condition \eqref{momsec1} of a compatible $E$-$n$-form $J$,
we modify~$Q_{{\rm Lie}\,J}$ by adding a term induced from a (pre)-$n$-plectic form $\omega=H$ as\footnote{In this section and the next section, a pre-$n$-plectic form on a (normal) base manifold $M$ is denoted by $H$ not~$\omega$ to distinguish with $\gomega$.}
\begin{align}
\Qadm ={}& \sbv{-\Theta_{\rm Lie} + (-1)^n \uJ}{-} + (-1)^{n(n+1)/2} \bigl(\gomega^{\flat}\bigr)^{-1} \iota_{\rho(q)}^n H
\nonumber \\
={}& \rho^i_a(x) q^a \frac{\partial}{\partial x^i}
- \frac{1}{2} C_{bc}^a(x) q^b q^c \frac{\partial}{\partial q^a}
\nonumber \\ &
+ \left((-1)^n \rho^i_a z_i
- C_{ab}^c(x) q^b y_c
+ (-1)^n \frac{1}{(n-1)!} J_{a b_2 \dots b_n} q^{b_2} \cdots q^{b_n} \right)
\frac{\partial}{\partial y_a}
\nonumber \\ &
- \left[ \partial_i\rho^j_a z_j q^a
- \frac{(-1)^n}{2} \partial_i C_{bc}^a(x) q^b q^c y_a
\right.
\nonumber \\ &
\left.
+ (-1)^n \frac{1}{n!} \bigl(\partial_i J_{a_1 \dots a_n}
- \rho^{j_1}_{a_1} \cdots \rho^{j_n}_{a_n}
H_{i j_1 \dots j_n} \bigr) q^{a_1}\cdots q^{a_n}
\right] \frac{\partial}{\partial z_i},
\label{homologicalvf}
\end{align}
where
\begin{align*}
H = \frac{1}{(n+1)!} H_{i_1 \dots i_{n+1}}(x) \rd x^{i_1} \wedge \cdots \wedge \rd x^{i_{n+1}},
\end{align*}
is a closed $(n+1)$-form.
For a $1$-form
\[\alpha = \alpha_i(x) \rd x^i,\qquad
\bigl(\gomega^{\flat}\bigr)^{-1}\colon \ T^*[n-1]E[1] \rightarrow TE[1]\]
is the map induced from the symplectic form $\gomega$.
\begin{align*}
\bigl(\gomega^{\flat}\bigr)^{-1} \alpha = \alpha_i(x) \frac{\partial}{\partial z_i}.
\end{align*}
\begin{Proposition}
Let $E$ be a Lie algebroid and $H$ is a pre-$n$-plectic form.
The vector field of degree $+1$, $\Qadm$ in \eqref{homologicalvf} satisfies
$\Qadm^2=0$ if and only if $J$ is a compatible $E$-$n$-form.
\end{Proposition}
\begin{proof}
We compute the square of $\Qadm$.
Using equations $Q_{{\rm Lie}\,J}^2=0$ and
\[\big\{\Theta_{\rm Lie},\iota_{\rho(q)}^n H\big\}=
\big\{\uJ,\iota_{\rho(q)}^n H\big\} = 0,\] etc.,
we obtain that $\Qadm^2$ is proportional to $\sbv{\Theta_{\rm Lie}}{\uJ} + \iota_{\rho(q)}^{n+1} H$,
which must be zero. This equation is equivalent to
the condition of a compatible $E$-$n$-form~\eqref{momsec1}.
\end{proof}

\begin{Example}
Let $\pi = \frac{1}{2} \pi^{ij}(x) \partial_i \wedge \partial_j
 \in \Gamma\bigl(\wedge^2 TM\bigr)$ be a bivector field and
$H = \frac{1}{3!} H_{ijk}(x) \rd x^i \wedge \rd x^j \wedge \rd x^k \in \Omega^3(M)$ be a closed $3$-form.
We take $n=2$ and consider an N-manifold $T^*[1]T[1]M$.
Local coordinates are $\bigl(x^i, q_i, z_i, y^i\bigr)$ of degree $(0, 1, 1, 0)$.
Now, the homological vector field $Q$ is given by
\begin{align*}
Q ={}& \pi^{ij}(x) q_i \frac{\partial}{\partial x^j}
- \frac{1}{2} \partial_i \pi^{jk}(x) q_j q_k \frac{\partial}{\partial q_i}
- \bigl(\pi^{ij} z_j
+ \partial_k \pi^{ij}(x) q_j y^k
- \pi^{ij} q_j \bigr)
\frac{\partial}{\partial y_i}
 \\
& + \left[\partial_i \pi^{jk} z_j q_k
+ \frac{1}{2} \partial_i \partial_j \pi^{kl} q_k q_l y_j
- \frac{1}{2} \bigl(\partial_i \pi^{jk}
- \pi^{jl} \pi^{km}
H_{ilm} \bigr) q_j q_k
\right] \frac{\partial}{\partial z_i},
\end{align*}
$Q^2=0$ is equivalent that $(\pi, H)$ is the twisted Poisson structure.
\end{Example}

There does not necessarily exist a homological function $\Theta$ satisfying
$\Qadm = \sbv{\Theta}{-}$ for the homological function in \eqref{homologicalvf},
if $H$ is nonzero. A compatible $E$-$n$-form can not be constructed as a QP-structure on $T^*[n-1]E[1]$.\footnote{Since a homological function for a compatible $E$-$n$-form does not exist, we cannot construct a topological sigma model by using the so called AKSZ construction \cite{Alexandrov:1995kv}.}

\section{Realization by QP-manifold and twist}\label{Qmfdtwist}
In the Q-manifold description in the previous section, a compatible condition is not formulated as a QP-manifold, i.e.,
there exists no
homological function $\Theta$ such that $\iota_{Q} \omega = - \delta \Theta$
for the homological vector field $Q$ in equation~\eqref{homologicalvf}
if $H \neq 0$.
We consider another realization in graded geometry as
a Lagrangian Q-submanifold of a QP-manifold.

For it, we consider a shifted cotangent bundle
of the graded bundle in Section~\ref{Qmfd}.
Take a~shifted double cotangent bundle $T^*[n]\calM$
of $\calM = T^*[n-1]E[1]$ in the previous Section~\ref{Qmfd}.
We consider a QP-structure on $T^*[n]\calM$, i.e., we take the canonical
graded symplectic form $\ggomega$ and \smash{$\tQ$} such that $\calL_{\tQ} \ggomega =0$.

\subsection{Twisting of homological functions and twisted QP-manifolds}

In this section, a general theory to describe a Lagrangian submanifold and
its geometric structure of a QP-manifold in terms of graded geometry
and applied to a compatible $E$-$n$-form.
We use twisting of a homological function introduced by
\cite{Roytenberg:2001am} and analyzed in
\cite{Ikeda:2013wh,Kosmann-Schwarzbach:2007,Terashima}.

Let $\proj\colon T^*[n]\calM \rightarrow \calM$ be a natural projection map.
$\calM$ is a trivial Lagrangian graded submanifold since
$\{\proj^* f, \proj^* g\} =0$ for all $f, g \in C^{\infty}(\calM)$
and $\calM$ is of half dimensions of~$T^*[n]\calM$.

If $\bigl(T^*[n]\calM, \ggomega, \tQ\bigr)$ is a QP-manifold of degree $n$,
an induced graded
Poisson structure of degree $-n$ is defined on $\calM$ as follows.
A bilinear graded bracket of degree $-n+1$ is defined on~$\calM$ by
a derived bracket,
\begin{align}
\{f, g \}_{\calM} &:= \proj_* \big\{\big\{\proj^* f, \Thetan \big\}, \proj^* g\big\},
\label{derivedbracket01}
\end{align}
where $\{-,-\}$ is the graded Poisson bracket induced from \smash{$\ggomega$}
and $\Thetan$ is a homological function for \smash{$\tQ$}.
$\{-, -\}_{\calM}$ is a graded Poisson bracket of degree $-n$ by using
$\{\proj^* f, \proj^* g \} =0$.
In fact, $\{-, - \}_{\calM}$ satisfies identities of a graded Poisson bracket,
\begin{align*}
&\{f, g \}_{\calM} = - (-1)^{(|f|-n+1)(|g|-n+1)} \{g, f \}_{\calM},
\\
&\{f, gh \}_{\calM} = \{f, g\}_{\calM} h + (-1)^{(|f|-n+1)|g|} g \{f, h \}_{\calM},
\\
&\{f, \{g, h \}_{\calM} \}_{\calM} = \{\{f, g \}_{\calM}, h \}_{\calM}
+ (-1)^{(|f|-n+1)(|g|-n+1)} \{g, \{f, h \}_{\calM} \}_{\calM}
\end{align*}
for $f, g, h \in C^{\infty}(\calM)$.

In order to construct a general geometric structure realized in a Lagrangian submanifold $\calL \subset T^*[n+1]\calM$, we use \textit{twist} of a homological function $\Thetan$.
\begin{Definition}
Let $\phi \in C^{\infty}(T^*[n]\calM)$.
Then the \textit{twist} ${\rm e}^{\operatorname{ad} \phi}$ of any function $\alpha \in C^{\infty}(\calM)$ by $\phi$ is defined by~\cite{Roytenberg:2001am}
\begin{align*}
{\rm e}^{\operatorname{ad} \phi} \alpha := \alpha + \{\alpha, \phi\} + \frac{1}{2}\{\{\alpha, \phi\}, \phi\}
+ \cdots + \frac{1}{n!}\{\dots\{\{\alpha, \phi\}, \phi\}, \dots, \phi \} + \cdots.
\end{align*}
\end{Definition}
If a function $\phi \in C^{\infty}(T^*[n]\calM)$ is of degree $n$, twist preserves degree and gives a Poisson map~${\alpha \mapsto {\rm e}^{\operatorname{ad} \phi} \alpha }$ satisfying
\begin{align*}
\big\{{\rm e}^{\operatorname{ad} \phi} \alpha, {\rm e}^{\operatorname{ad} \phi} \beta \big\}
= {\rm e}^{\operatorname{ad} \phi} \{\alpha, \beta \}
\end{align*}
for all $\alpha, \beta \in C^{\infty}(T^*[n]\calM)$.

Let $\calE_{\phi} := \big\{ {\rm e}^{\operatorname{ad} \phi} \proj^* f \mid f
\in C^{\infty}(\calM) \big\}$ be a space of functions twisted
by a degree $n$ function $\phi \in C^{\infty}(T^*[n]\calM)$.
In general, $\calE_{\phi}$ is identified as a space of functions on
some graded manifold
$\calL_{\phi} \subset T^*[n]\calM$, $\calE_{\phi} = C^{\infty}(\calL_{\phi})$,
by applying a Serre--Swan type theorem on a graded manifold.
\begin{Proposition}
$\calL_{\phi}$ is a graded Lagrangian submanifold of $T^*[n]\calM$.
\end{Proposition}
It is confirmed using the formula
$\big\{{\rm e}^{\operatorname{ad} \phi} \proj^*f, {\rm e}^{\operatorname{ad} \phi} \proj^* g\big\}
= {\rm e}^{\operatorname{ad} \phi} \{\proj^* f, \proj^* g \} =0$
and the dimension of $\calL_{\phi}$ is the same as $\calM$.

By definition, for $\alpha, \beta \in C^{\infty}(\calL_{-\phi})$,
there exist $f, g \in C^{\infty}(\calM)$ such that $\alpha = {\rm e}^{-\operatorname{ad} \phi} \proj^* f$ and $\beta = {\rm e}^{-\operatorname{ad} \phi} \proj^* g$.
The derived bracket of $\alpha$ and $\beta$ is equal to
\begin{align}
\big\{\big\{{\rm e}^{-\operatorname{ad} \phi} \proj^* f, \Thetan\big\}, {\rm e}^{-\operatorname{ad} \phi} \proj^* g\big\} &:= {\rm e}^{-\operatorname{ad} \phi} \big\{\big\{\proj^* f, {\rm e}^{\operatorname{ad} \phi} \Thetan\big\}, \proj^* g\big\}.
\label{derivedbracket02}
\end{align}
A bilinear bracket is defined by the following derived bracket,
\begin{align}
\{\alpha, \beta \}_{\calL_{-\phi}}
&:= \proj_* {\rm e}^{-\operatorname{ad} \phi} \big\{\big\{\proj^* f, {\rm e}^{\operatorname{ad} \phi} \Thetan\big\}, \proj^* g\big\}.
\label{derivedbracket03}
\end{align}
Since it is equivalent to \smash{$\proj_* \big\{\big\{{\rm e}^{-\operatorname{ad} \phi} \proj^* f, \Thetan\big\}, {\rm e}^{-\operatorname{ad} \phi} \proj^* g\big\}$}, then twisting by $\phi$ gives
a graded Poisson bracket on $\calL_{-\phi}$ induced by \smash{$\Thetan$} equivalent to
a graded Poisson bracket on $\calM$ induced by~${\rm e}^{\operatorname{ad} \phi} \Thetan$.
We can prove that $\{-, -\}_{\calL_{-\phi}}$ is a graded Poisson bracket of degree $-n+1$ similar to the proof for the bracket \eqref{derivedbracket01}.

Thus, in the preceding discussion, we analyze the twisted homological function ${\rm e}^{\operatorname{ad} \phi} \Thetan$ on $\calM$ instead of a homological function on $\calL_{-\phi}$.

The twisting function $\phi$ is an almost homological function.
By \eqref{derivedbracket02} and \eqref{derivedbracket03},
the graded Poisson bracket is defined by the derived bracket
$\big\{\big\{-, \Thetan\big\},-\big\}$.
On the other hand, we obtain
\begin{align}
&\proj_* {\rm e}^{\operatorname{ad} \phi} \Thetan\nonumber\\
&\qquad=
\proj_* \Big(\Thetan + \big\{\Thetan, \phi\big\} + \frac{1}{2}\big\{\big\{\Thetan, \phi\big\}, \phi\big\}
+ \cdots + \frac{1}{n!} \underbrace{\big\{\dots\big\{\big\{\Thetan, \phi\big\}, \phi\big\}, \dots, \phi \big\}}_{n\, \text{brackets}} + \cdots \Big).
\label{twistdQP}
\end{align}
The function $\phi$ is of degree $n$, and
from equation \eqref{derivedbracket01},
the third term of $\proj_* \, {\rm e}^{\operatorname{ad} \phi} \Theta$
gives
\begin{align}
& \{\phi, \phi \}_{\calM} = - \proj_* \big\{\big\{\Thetan, \phi\big\}, \phi\big\}.
\label{db02}
\end{align}
If equation \eqref{db02} is zero, $\phi$ is a homological function on $\calM$
and
$(\calM, \{-,-\}_{\calM}, \phi)$ is a QP-manifold of degree $n-1$.
However, equation \eqref{db02} is not necessarily zero.

Inspired by the above analysis and equation \eqref{twistdQP},
we define a \textit{twisted QP-structure}.
We consider the following class of generalizations of a QP-manifold which has good properties but it is not a QP.
\begin{Definition}
Take a shifted cotangent bundle $\calN = T^*[n]\calM$ of a graded manifold $\calM$. Let $\calN = T^*[n]\calM$ be a QP-manifold with a homological function $\Theta$.
If a degree $n$ function $\phi \in C^{\infty}(T^*[n]\calM)$ satisfies
\begin{align}
\proj_* {\rm e}^{\operatorname{ad} \phi} \Theta = 0,\label{LagQ}
\end{align}
$(\calM, \{-,-\}_{\calM}, \phi)$ is called a \textit{twisted QP-manifold},
where $\proj\colon T^*[n]\calM \rightarrow \calM$ is a natural projection.
\end{Definition}
A function $\phi$ satisfying the condition \eqref{LagQ} is also called a
a Poisson function or a canonical function.
Obviously a QP-manifold is a twisted QP-manifold.
A twisted QP-manifold has some applications to physical theories
such as the reduction of the AKSZ sigma models and the supergeometric construction of current algebras~\cite{ Bessho:2015tkk, Ikeda:2013vga, Ikeda:2013wh}.

\subsection{QP-manifold description of compatible differential forms}\label{QPcompatible}

Our graded manifold is $T^*[n]\calM = T^*[n]T^*[n-1]E[1]$.
In order to compute the concrete formula, take local coordinates $\bigl(x^i, q^a, z_i, y_a\bigr)$ of degree $(0, 1, n-1, n-2)$
on $T^*[n-1]E[1]$ and canonical conjugate coordinates of the fiber
$\bigl(\xi_i, p_a, \zeta^i, \eta^a\bigr)$ of degree $(n, n-1, 1, 2)$.
The canonical projection~${\proj\colon T^*[n]\calM \rightarrow \calM}$ is given by
$\proj\colon \bigl(\xi_i, p_a, \zeta^i, \eta_a\bigr) \mapsto (0,0,0,0)$.
The canonical symplectic form of degree $n$ is
\begin{align*}
\ggomega &= \delta x^i \wedge \delta \xi_i + \delta q^a \wedge \delta p_a
+ \delta z_i \wedge \delta \zeta^i + \delta y_a \wedge \delta \eta^a.
\label{ggomega}
\end{align*}
Graded Poisson brackets induced from the above symplectic form are
\begin{gather*}
\big\{x^i, \xi_j \big\} = - \big\{\xi_j, x^i \big\} = \delta^i_j,
\qquad
\{q^a, p_b \} = - (-1)^{n-1} \{p_b, q^a \} = \delta^a_b,
\\
\big\{z_i, \zeta^j \big\} = - (-1)^{n-1} \big\{\zeta^j, z_i \big\} = \delta_i^j,
\qquad
\big\{y_a, \eta^b \big\} = - \big\{\eta^b, y_a \big\} = \delta_a^b.
\end{gather*}
We take an $(n+1)$-form $H$ on $M$ and
consider the following function of degree $n+1$,
\begin{align*}
\Thetan &= \bracket{\xi}{\zeta} + \inner{p}{\eta} + (-1)^{\tfrac{n(n-1)}{2}}
\tilde{H}
 \\ &=
\xi_i \zeta^i + p_a \eta^a + (-1)^{\tfrac{n(n-1)}{2}} \frac{1}{(n+1)!} H_{i_1 \dots i_{n+1}}(x) \zeta^{i_1} \cdots \zeta^{i_{n+1}},
\end{align*}
where
\smash{$\tilde{H} = \frac{1}{(n+1)!} H_{i_1 \dots i_{n+1}}(x)
\zeta^{i_1} \cdots \zeta^{i_{n+1}}$} is a degree $n+1$ function corresponding to
the $(n+1)$-form \smash{$H = \frac{1}{(n+1)!} H_{i_1 \dots i_{n+1}}(x) \rd x^{i_1} \wedge \cdots \wedge \rd x^{i_{n+1}} \in \Omega^{n+1}(M)$}.
An $(n+1)$-form is identified to a~degree $n+1$ function on $T[1]M$.
$\Thetan$ satisfies \smash{$\big\{\Thetan,\Thetan\big\}=0$} if and only if $H$ is closed, $\rd H=0$,
i.e.,
\smash{$(T^*[n]\calM, \ggomega, \tQ)$} is a QP-manifold if $H$ is a pre-$n$-plectic form on $M$,
where $\tQ$ is a Hamiltonian vector field for $\Thetan$.

We consider a twist using the following function of degree $n$,
$\phi= \phiL - \uJ$, where two functions are
\begin{align}
&\phiL = \rho^i_a(x) z_i q^a + \frac{1}{2} C^c_{ab}(x) q^a q^b y_c,
\label{twist1}
\\
&\uJ = \frac{1}{n!} J_{a_1\dots a_n}(x) q^{a_1}\cdots q^{a_n}.
\label{twist2}
\end{align}
and impose the condition \eqref{LagQ}, $\proj_* {\rm e}^{\operatorname{ad} \phi} \Thetan = 0$.
$\phiL$ is a similar form as $\Theta_{\rm Lie}$ in equation~\eqref{LAHamiltonian}.

Compare to the condition in non-graded geometry
by defining the following operations.
Let~${\rho := \rho^i_a(x) \tfrac{\partial}{\partial x^i}\colon E \rightarrow TM}$ be a bundle map,
$[e_a, e_b] := C_{ab}^c e_c$ be a bilinear bracket on $\Gamma(E)$
for the basis $e_a \in \Gamma(E)$.
Moreover let $J := \frac{1}{n!} J_{a_1\dots a_n}(x) {\rm e}^{a_1} \wedge \cdots \wedge {\rm e}^{a_n} \in \Gamma(\wedge^n E^*)$
be an $E$-$n$-form with the basis ${\rm e}^a \in \Gamma(E^*)$.

By structures of coordinates, $z_i$, $q^a$ and $y_a$,
only nonvanishing terms in $\proj_* {\rm e}^{\operatorname{ad} \phi} \Thetan$
after the projection $\proj$ are following four terms,
\begin{gather*}
\big\{\big\{\Thetan, \phiL \big\}, \phiL \big\},\quad
\big\{\big\{\Thetan, \uJ \big\}, \phiL \big\},\quad \big\{\big\{\Thetan, \phiL \big\}, \uJ \big\}
,\quad
\underbrace{\big\{\dots\big\{\big\{\Thetan, \phiL \big\}, \phiL \big\}, \dots, \phiL
\big\}}_{(n+1) \text{brackets}}.
\end{gather*}
At first, \smash{$\big\{\big\{\Thetan, \phiL \big\}, \phiL \big\}=0$} is satisfied independently.
This equation gives conditions for $\rho^i_a$ and $C^a_{bc}$ that
$\rho_a = \rho^i_a \partial_i$ is an anchor map and $C^a_{bc}$
is a structure function of the Lie algebroid on~$E$.
Another condition
\begin{gather}
\proj_* \bigg( \frac{1}{2} \big\{\big\{\Thetan, \uJ \big\}, \phiL \big\}
+ \frac{1}{2} \big\{\big\{\Thetan, \phiL \big\}, \uJ \big\}\nonumber\\
\qquad
{}+ \frac{1}{(n+1)!}
\underbrace{\{\dots\{\{\Theta, \phiL \}, \phiL \}, \dots, \phiL \}}_{(n+1)\, \text{brackets}} \bigg) =0
\label{twisteQP01}
\end{gather}
gives the compatible condition of $J$, equation \eqref{momsec1}.
In fact, the first two terms in \eqref{twisteQP01} is equivalent to
${}^E \rd J$ and the third term is $\iota_{\rho}^{n+1} H$
by straightforward calculations.

From these analysis, we obtain the following result.
\begin{Proposition}
$(E, \rho, [-,-])$ is a Lie algebroid and $J$ is a compatible $E$-$n$-form
with a pre-$(n+1)$-form $H$
if and only if $\phi$ satisfies equation \eqref{LagQ}
$\proj_* {\rm e}^{\operatorname{ad} \phi} \Thetan = 0$, i.e., $\calM$ is a \textit{twisted QP-manifold} with $\phi$ in equations \eqref{twist1} and \eqref{twist2}.
\end{Proposition}

\subsection*{Acknowledgments}

The author is grateful to the Erwin Schr\"odinger International Institute for Mathematics and Physics for support within the program ``Higher Structures and Field Theory'' in 2022, and National Center for Theoretical Sciences and National Tsing Hua University, where part of this work was carried, for their hospitality. This work was supported by JSPS Grants-in-Aid for Scientific Research Number 22K03323.
I would like to thank to Hsuan-Yi Liao, Camille Laurent-Gengoux and Seokbong Seol for his hospitality and useful discussion. Especially, he would like to thank to the anonymous referees for relevant contribution to improve the paper.

\pdfbookmark[1]{References}{ref}
\LastPageEnding

\end{document}